\documentclass[12pt]{amsart}
\usepackage{amsthm}
\usepackage{amsmath}
\usepackage{amssymb}
\usepackage[dvipsnames]{xcolor}
\usepackage[hidelinks]{hyperref}
\usepackage{cleveref}

\hypersetup{
    colorlinks=true,
    linkcolor=blue,
    urlcolor=red,
    citecolor=Green
    }

\advance\paperheight1.00cm
\advance\textheight1.00cm
\advance\vsize1.00cm
\advance\topskip-0.5cm
\advance\voffset-0.5cm
\DeclareTextFontCommand{\emph}{\color{Blue}\em}

\newtheorem{theorem}{Theorem}[section]

\newtheorem{conjecture}[theorem]{Conjecture}

\newtheorem{lemma}[theorem]{Lemma}
\newtheorem{corollary}[theorem]{Corollary}

\theoremstyle{definition}

\newtheorem{definition}[theorem]{Definition}

\theoremstyle{remark}
\newtheorem*{remark}{Remark}

\renewcommand{\S}{\mathfrak{S}}
\newcommand{\C}{\mathbb{C}}
\newcommand{\x}{\mathbf{x}}
\newcommand{\y}{\mathbf{y}}
\renewcommand{\t}{\mathbf{t}}

\newcommand{\Z}{\mathbb{Z}}
\newcommand{\G}{\mathfrak{G}}

\title{Graham positivity of triple Schubert calculus}
\author{Yibo Gao}
\address{Beijing International Center for Mathematical Research, Peking University, Beijing 100871, China}
\email{gaoyibo@bicmr.pku.edu.cn}
\author{Rui Xiong}
\address{Department of Mathematics and Statistics, University of Ottawa, 150 Louis-Pasteur, Ottawa, ON, K1N 6N5, Canada}
\email{rxion043@uOttawa.ca}

\date{\today}

\begin{document}
\pagestyle{plain}
\begin{abstract}
We prove Samuel's conjecture on certain Graham positivity of the expansion coefficient of two double Schubert polynomials in three sets of variables by establishing a refined version of Graham's positivity theorem. As a corollary, we prove Kirillov's conjecture on the positivity of skew divided difference operators applied to Schubert polynomials. 
\end{abstract}
\maketitle

\section{Introduction}\label{sec:intro}
Double Schubert polynomials $\{\S_w(\x;\y)\:|\: w\in S_n\}$, indexed by permutations, are polynomial representatives of Schubert classes in the torus-equivariant cohomology ring of the flag variety $\mathrm{Fl}_n(\C)$, where $\x=(x_1,\ldots,x_n)$ and $\y=(y_1,\ldots,y_n)$ are sequences of variables. They can be defined via divided difference operators, and enjoy rich algebraic and combinatorial properties. In particular, they can be computed via different combinatorial models including pipe dreams \cite{bergeron-billey,knutson-miller} and bumpless pipe dreams \cite{LLS}. 

In this paper, we focus on the coefficients $c_{u,v}^w(\y,\t)$ from the expansion 
\begin{equation}\label{eq:main}
\S_{u}(\x;\y)\cdot\S_{v}(\x;\t)=\sum_{w\in S_{\infty}}c_{u,v}^w(\y,\t)\cdot \S_{w}(\x;\t).
\end{equation}
In the case $\y=\t=\mathbf{0}$, $c_{u,v}^w:=c_{u,v}^w(\mathbf{0},\mathbf{0})\in\mathbb{N}$ is the \emph{Schubert structure constant} which is the core to Hilbert's fifteenth problem, and finding combinatorial models for these coefficients has sparkled tremendous interest throughout the years in the community of algebraic combinatorics and algebraic geometry. In the case where $u$ and $v$ are $k$-Grassmannian permutations, $\S_{u}(\x;\y)$ and $\S_{v}(\x;\t)$ are the \emph{factorial Schur polynomials}; Molev-Sagan \cite{molev-sagan} first provided a combinatorial formula for these coefficients and Knutson-Tao \cite{knutson-tao-triple} provided geometric meanings to the expansion. 

Our main result is the following:
\begin{theorem}[{\cite[Conjecture~1.1]{sammuel}}]
\label{thm:main}
For $u,v,w\in S_\infty$, 
$c_{u,v}^w(\y,\t)\in\mathbb{N}[t_i-y_j]_{i,j\geq 1}$.
\end{theorem}
Besides the above-mentioned situations, a well-studied case of \Cref{thm:main} is that of \emph{separated descents} (\cite{huang-separated-descents,knutson-zinnjustin-separated-descents}), i.e. $\max\mathrm{Des}(u)\leq\min\mathrm{Des}(v)$, where Samuel \cite{sammuel} established a positive formula, and Fan, Guo and the second author \cite{fan-guo-xiong} found another formula using puzzles in the generality of double Grothendieck polynomials. 

Setting $\y=\mathbf{0}$, we obtain Kirillov's conjecture:
\begin{corollary}[{\cite[Conjecture~1]{kirillov}}]
\label{cor:kirillov}
For $u,v,w\in S_n$, $\partial_{w/v}\S_u(\x)\in\mathbb{N}[\x]$.
\end{corollary}

The skew divided difference operators $\partial_{w/v}$, first introduced by Macdonald \cite{macdonald_notes}, provide a method for computing Schubert structure constants. These operators have attracted significant attention due to their connections to Fomin-Kirillov algebras \cite{fomin-kirillov}
. Their generalizations and positivity properties have been extensively studied in works including \cite{barligea-2018, berenstein-richmond-2015, liu-2015}.

Our proof relies on a refined version of Graham's positivity theorem \cite{graham-positivity} (\Cref{thm:graham-positivity-refined}). To apply this result, we establish a new geometric interpretation (\Cref{sub:triple-proof}) for the coefficients \(c_{u,v}^w(\mathbf{y},\mathbf{t})\), distinct from the one given by Knutson-Tao \cite{knutson-tao-triple}. A byproduct of this theorem is a geometric explanation for the positivity in Billey’s formula (\Cref{sub:billey}).

We end our discussion with a conjecture for Grothendieck polynomials (\Cref{sub:conjecture}).

\section{Proof of the main theorem}\label{sec:proof}
\subsection{A refined Graham positivity}
Let $G$ be a connected, complex, reductive algebraic group, $B\subset G$ be a Borel subgroup, $B^-$ be its opposite Borel subgroup, $T=B\cap B^-$ be a maximal torus, $N$ (resp., $N^-$) be the unipotent radical of $B$ (resp., $B^-$), $W=N_G(T)/T$ be the Weyl group and $\Phi$ be the associated root system, with positive roots $\Phi^+$ and simple roots $\Delta=\{\alpha_1,\ldots,\alpha_r\}$. For $\alpha\in\Phi$, write $s_{\alpha}$ for the corresponding reflection, and write $s_{i}$ for $s_{\alpha_i}$ for simplicity. For $w\in W$, its (left) \emph{inversion set} is $I(w):=\{\alpha\in\Phi^+\:|\: w^{-1}\alpha\in\Phi^-\}$ and its \emph{non-inversion set} is $J(w):=\{\alpha\in\Phi^+\:|\: w^{-1}\alpha\in\Phi^+\}=I(ww_0)$, where $w_0$ is the longest element in $W$. 
For each $w\in W$, we pick a representative of it in $N_G(T)$, and also denote it by $w$, slightly abusing notation. 

The \emph{flag variety} $G/B$ admits a \emph{Bruhat decomposition} $\bigsqcup_{w\in W}B^- w B/B$ into \emph{Schubert cells}. Their closures $\overline{B^-w B/B}\subseteq G/B$ are the \emph{Schubert varieties}. The \emph{Schubert classes} $\{[\overline{B^- w B/B}]_T\:|\: w\in W\}$
form an $H_T^*(\mathsf{pt})$-basis of $H_T^*(\mathrm{Fl}_m(\C))$.

We define the following closed subgroups of $G$.

\begin{definition}
For $w\in W$, define $N^-(w):=N^-\cap wN^- w^{-1}$ and $B^-(w)=T\cdot N^-(w)$.
\end{definition}

By \cite[Section 28]{humphreys_algebraic_group}, as a variety, $N^-(w)$ is isomorphic to the Schubert cell $B^-wB/B\cong \mathbb{C}^{\ell(w_0)-\ell(w)}$ via $x\mapsto xwB/B$. In particular, $N^-(w)$ is connected. Its Lie algebra is 
$$\mathfrak{n}^-(w)=\operatorname{span}(F_{\alpha}\:|\: \alpha\in J(w))\subseteq \mathfrak{g}:=\operatorname{Lie}G$$
where $F_\alpha$ is a root vector of weight $-\alpha$.

\begin{lemma}\label{lem:normal-subgroup}
If $ws_i>w$, then $N^-(ws_i)$ is a normal subgroup of $N^-(w)$.
\end{lemma}
\begin{proof}
Recall that $[F_{\alpha},F_{\beta}]\in\C^{\times}F_{\alpha+\beta}$ if $\alpha+\beta$ is a root, and $[F_{\alpha},F_{\beta}]=0$ otherwise. As $ws_i>w$, we have $J(w)=J(ws_i)\cup\{w\alpha_i\}$ and thus $\mathfrak{n}^-(ws_i)\subset \mathfrak{n}^-(w)$, $N^-(ws_i)$ is a closed subgroup of $N^-(w)$ by \cite[Theorem 13.1]{humphreys_algebraic_group}. In fact, $\mathfrak{n}^-(ws_i)$ is an ideal of $\mathfrak{n}^-(w)$. To check this, take any $\alpha\in J(ws_i)$ and $\beta\in J(w)$, and it suffices to show $[F_{\alpha},F_{\beta}]\in \mathfrak{n}^-(ws_i)$. If $\alpha+\beta\notin\Phi$, $[F_{\alpha},F_{\beta}]=0$. Thus, consider $\gamma=\alpha+\beta\in\Phi^+$. If $\beta\in J(ws_i)$, then $\gamma\in J(ws_i)$ as well; and if $\beta=w\alpha_i$, $\gamma\neq w\alpha_i\in J(w)$, so $\gamma\in J(ws_i)$. Indeed, $\mathfrak{n}^-(ws_i)$ is an ideal of $\mathfrak{n}^-(w)$ and by \cite[Theorem 13.3]{humphreys_algebraic_group}, $N^-(ws_i)$ is a normal subgroup of $N^-(w)$. 
\end{proof}

We are now in a position to state a refined version of Graham positivity theorem. 

\begin{theorem}\label{thm:graham-positivity-refined}
Let $B^-$ act on a non-singular variety $X$, and let $Y$ be a $B^-(w)$-invariant effective cycle in $X$. Then there exist $B^-$-invariant effective cycles $Z_1,\ldots,Z_m$ such that 
\[[Y]_T\in \sum_{i=1}^m\mathbb{N}[-\alpha]_{\alpha\in I(w)}\cdot [Z_i]_T\]
in $H_T^*(X)$. 
\end{theorem}
\begin{proof}
We use induction on $\ell(w)$. 
When $w=\mathsf{id}$, $B^-(w)=B^-$, so there is nothing to prove. 
Assume the theorem holds for $w$. 
Let us prove the statement for $ws_i>w$.  
By \Cref{lem:normal-subgroup}, the pair $B^-(ws_i)\subset B^-(w)$ satisfies the condition (see \cite[p. 384]{AF}) of \cite[Proposition 19.4.4]{AF} with 
$\chi = -w\alpha_i$. So there exists $B^-(w)$-invariant effective cycles $Z_1$ and $Z_2$ such that 
$[Y]_T=[Z_1]_T+\chi [Z_2]_T.$
Since $I(ws_i)=I(w)\cup \{w\alpha_i\}$, the inductive step is now established.
\end{proof}

\begin{corollary}\label{coro:G/B}
Let $X=G/B$ be the flag variety. Under the above setting, we have 
\[[Y]_T=\sum_{w\in W}\mathbb{N}[-\alpha]_{\alpha\in I(w)}\cdot [\overline{B^-wB/B}]_T\]
in $H^*_T(G/B)$. 
\end{corollary}
\begin{proof}
Since flag variety has finite many $B^-$-orbits, any $B^-$-invariant effective cycle over $G/B$ must be a non-negative combination of Schubert classes $[\overline{B^-wB/B}]_T$. 
\end{proof}

\begin{remark}
When $w=w_0$ is the longest element, $B^-(w_0)=T$ and \Cref{thm:graham-positivity-refined} gives Graham's positivity theorem \cite[Theorem 3.2]{graham-positivity}.
Our assumption is stronger than that of Graham's, since a $B^-(w)$-invariant cycle is necessarily $T$-invariant, yielding a stronger positivity result where roots are restricted to $I(w)$ rather than all positive roots.
\end{remark}

\subsection{A geometric explanation of positivity in Billey's formula}\label{sub:billey}
As an application of \Cref{thm:graham-positivity-refined}, we give a geometric explanation of positivity in Billey's formula \cite{Billey-formula}. See \cite{Tymoczko-Billey-survey} for a great survey on this subject.

The $T$-fixed points of $G/B$ are $(G/B)^T=\{wB/B\:|\:w\in W\}$. 
Define the localization to be the restriction map
$\cdot|_w: H_T^*(G/B)\to H_T^*(wB/B)\simeq H_T^*(\mathsf{pt}).$
Billey's formula \cite{Billey-formula} is a combinatorial formula of the localization of Schubert classes $[\overline{B^-uB/B}]_T|_w$ for any $u,w\in W$. From its explicit form, which we do not provide here, we see that 
\begin{equation}\label{eq:Billey+}
{}[\overline{B^-uB/B}]_T|_{w} \in \mathbb{N}[\alpha]_{\alpha\in I(w)}.\end{equation}
We provide a geometric explanation of this positivity using \Cref{coro:G/B}.

Recall that the point $w_0B/B=B^-w_0B/B$ is invariant under $B^-$.
So the torus fixed point $ww_0B/B$ is invariant under $wB^-w^{-1}$. 
This implies that $ww_0B/B$ is invariant under $B^-(w)$. 
By \Cref{coro:G/B}, we have
$$[ww_0B/B]_T \in \sum_{u\in W} \mathbb{N}[-\alpha]_{\alpha\in I(w)}\cdot [\overline{B^-uB/B}]_T.$$
Following \cite[Section 3]{MNS-left-action} and \cite[Section 16.5]{AF}, for a Weyl group element $w\in W$, the automorphism $gB\mapsto wgB/B$ induces a left action $w^L: H_T^*(G/B)\to H_T^*(G/B)$. We remark that $w^L$ is not $H_T^*(\mathsf{pt})$-linear unless $w=\mathsf{id}$, but it is  semilinear with respect to the automorphism of $H_T^*(\mathsf{pt})$ induced by $w$. Applying $w_0^L$ to the above, we get 
$$[w_0ww_0B/B]_T \in \sum_{u\in W} \mathbb{N}[-w_0\alpha]_{\alpha\in I(w)}\cdot [\overline{Bw_0uB/B}]_T.$$
As $I(w_0ww_0)=\{-w_0\alpha\:|\:\alpha\in I(w)\}$, replacing $w_0ww_0$ by $w$ and $w_0u$ by $u$, we can rewrite 
\begin{equation}\label{eq:point-to-Schubert}
[wB/B]_T \in \sum_{u\in W} \mathbb{N}[\alpha]_{\alpha\in I(w)}\cdot [\overline{BuB/B}]_T.
\end{equation}
Now let us take the Poincar\'e pairing with $[\overline{B^-uB/B}]_T$ on both sides. 
The left-hand side is $[\overline{B^-uB/B}]_T|_{w}$, and the right-hand side is the coefficient of $[\overline{BuB/B}]_T$ in \eqref{eq:point-to-Schubert} 
by \cite[Proposition 7.3]{AF}.
This gives \eqref{eq:Billey+}.

\subsection{Application to triple Schubert calculus}\label{sub:triple-proof}
We restrict to the case of $G=\mathrm{GL}_{m}(\C)$. By \cite[Theorem 10.6.4]{AF}, the class $[\overline{B^-\pi B/B}]_T$ is represented by the \emph{double Schubert polynomial} $\S_{\pi}(\x;\t)$. We will not be working with the algebraic definition of $\S_{\pi}(\x;\t)$'s. For any permutation $\pi\in S_m$, we naturally identify it as $\pi\in S_m\hookrightarrow S_{\infty}$ via $\pi(k)=k$ for all $k>m$.

Fix permutations $u,v$. By picking $n\gg 0$, we can assume \Cref{eq:main} only involves those $w\in S_n$.
Let $G:=\mathrm{GL}_{2n}(\C)$ and $B,B^-,T$ be as above. We identify $H_{T}^*(\mathsf{pt})=\Z[t_1,\ldots,t_{2n}]$ with $t_i=-\epsilon_i:=-c_1(\C_{\epsilon_i})$, 
where \(\epsilon_{i}\) is the character of \(T\) corresponding to the \(i\)-th diagonal entry.  
We rename the variables $t_{n+i}=y_i$ for all $i\in[n]:=\{1,2,\ldots,n\}$. Consider a special permutation $\tau\in S_{2n}$ such that $\tau(i)=n+i$, $\tau(n+i)=i$ for all $i\in[n]$. By \cite[Section 16.5]{AF}, $[\tau\overline{B^-uB/B}]_T$ is represented by $\S_u(\x;\tau\t)=\S_u(\x;\y)$.

\begin{lemma}\label{lem:transversality}
The intersection $\tau\overline{B^-uB/B} \cap\overline{B^-vB/B}$ is proper and transverse at the generic point.
\end{lemma}
\begin{proof}
Let $w_0^{(m)}\in S_{m}$ be the longest permutation $m\ m{-}1\cdots 1$. For permutations $w,\pi\in S_n$, write $w\times \pi\in S_{2n}$ as the direct sum of $w$ and $\pi$; that is, $w\times\pi(i)$ is $w(i)$ if $i\leq n$, and is $\pi(i-n)+n$ if $i>n$. Since $v\in S_n$, $s_iv>v$ for $n\leq i<2n$, and thus $\overline{B^-vB/B}$ is invariant under $s_i$, where $s_i=(i\ i{+}1)$ is the simple transposition. Therefore, we have $\overline{B^-vB/B}=u_0\overline{B^-vB/B}$ where $u_0=1\times w_0^{(n)}$. Similarly, $\tau\overline{B^-uB/B}=\tau u_0\overline{B^-uB/B}$. As $u_0^{-1}\tau u_0=w_0^{(2n)}$, the lemma follows from \cite[Section 19.3]{AF}.
\end{proof}

We are now ready to prove our main theorem.
\begin{proof}[Proof of \Cref{thm:main}]
By \Cref{lem:transversality}, we can rewrite the coefficients of interest via \[[\tau\overline{B^-uB/B}\cap\overline{B^-vB/B}]_T=\sum_{w\in S_{n}}c_{u,v}^w(\y,\t)\cdot [\overline{B^-wB/B}]_T.\]
Since $\tau\overline{B^-uB/B}$ is closed under $\tau N^-\tau^{-1}$ and $\overline{B^-vB/B}$ is closed under $N^-$, the intersection is closed under $N^-\cap \tau N^-\tau^{-1}=:N^-(\tau)$. By \Cref{coro:G/B}, we conclude that $c_{u,v}^w(\y,\t)\in\mathbb{N}[-\alpha]_{\alpha\in I(\tau)}$, where we compute $I(\tau)=\{y_j-t_i\:|\:1\leq i,j\leq n\}$.
\end{proof}

Next we recall some definitions on (skew) divided difference operators.
\begin{definition}[{\cite{macdonald_notes}}]\label{def:skew}
The \emph{skew divided difference operator} $\partial_{w/u}$ is characterized by \[\partial_w(fg)=\sum_{v}\partial_{w/v}(f)\partial_v(g)\]
for all polynomials $f$ and $g$. Here, $\partial_w:=\partial_{i_1}\cdots\partial_{i_{\ell}}$ for any reduced word $w=s_{i_1}\cdots s_{i_{\ell}}$, and $\partial_i(f)=(f-s_i f)/(x_i-x_{i+1})$ is the \emph{divided difference operator}. 
\end{definition}
See \cite[Definition 4]{kirillov} for a more explicit definition of $\partial_{w/u}$.
\begin{proof}[Proof of \Cref{cor:kirillov}]
By \cite[Section 2 p. 5]{sammuel} or \cite[Proposition 6.4]{fan-guo-xiong}, $\partial_{w/v}\S_u(\x;\y)=c_{u,v}^w(\y,\x)$. Setting $\y=\mathbf{0}$, we obtain that $\partial_{w/v}\S_u(\x)=c_{u,v}^w(\mathbf{0},\x)\in\mathbb{N}[\x]$ by \Cref{thm:main}.
\end{proof}

\subsection{A positivity Conjecture for Grothendieck polynomials}\label{sub:conjecture}
Let $\beta$ be a formal variable and let $\G_w(\x;\y)$ be the \emph{double Grothendieck polynomial} for $w\in S_n$. We use the convention consistent with that of \cite{LLS-Ktheory}. Consider the expansion \[\G_u(\x;\y)\cdot\G_v(\x;\t)=\sum_{w\in S_\infty}\tilde{c}_{u,v}^{\,w}(\y,\t)\cdot \G_w(\x;\t).\]
We propose the following conjecture, inspired by \Cref{thm:main}.
\begin{conjecture}
For $u,v,w\in S_\infty$, $\tilde{c}_{u,v}^{\,w}(\y,\t)\in\mathbb{N}[\beta][t_i\ominus y_j]_{i,j\geq 1}$, where \(a\ominus b :=\frac{a-b}{1+\beta b}\). 
\end{conjecture}

\section*{Acknowledgements}
We thank Hai Zhu for enlightening conversations. Y.G. is partially supported by NSFC Grant No. 12471309.

\bibliographystyle{plain}
\bibliography{ref}
\end{document}